\documentclass[10pt]{article}
\usepackage{graphicx}
\usepackage[all]{xy}
\input xy
\usepackage{amsmath}
\usepackage{amsfonts}
\usepackage{amssymb}
\usepackage{url}
\usepackage{multirow}
\usepackage{mathabx}
\usepackage{dutchcal}
\usepackage{enumerate}
\usepackage{enumitem}
\usepackage{anysize}
\usepackage{relsize}
\usepackage[english]{babel}
\usepackage{hyperref}
\usepackage{accents}
\usepackage{makecell}
\usepackage{color}
\usepackage{float}

\newtheorem{remark}{Remark}

\newtheorem{theorem}{Theorem}

\newenvironment{proof}{\noindent{\bf Proof.}}%
{\hspace*{\fill}$\Box$\par\vspace{4mm}}

\newcommand{\kn}[2] {#1^{\underline{#2}}}
\newcommand{\col} {\mathbf{c}}

\newcommand{\pr} {{\rm Pr}}
\newcommand{\pra}[1] {\pr\left\{#1\right\}}
\newcommand{\pd}{\mathbb{P}}
\newcommand{\E} {\mathbb{E}}
\newcommand{\pt} {\mathcal{P}}
\newcommand{\Ea}[1] {\E\left(#1\right)}
\newcommand{\var}[1] {{\mathbb{V}ar}\left(#1\right)}

\newcommand{\diagonal}[1] {{\rm diag}(#1)}

\newcommand{\am}[1] {\overline{#1}}
\newcommand{\R} {\mathbb{R}}

\newcommand{\caporali}[1] {``#1''}
\newcommand{\mybreak} {\par\vspace{2mm}\noindent}

\def\cadre{$$\vcenter\bgroup\advance\hsize by -2em\noindent
	\refstepcounter{equation}(\theequation)~\ignorespaces}
\makeatletter
\def\endcadre{\egroup\eqno$$\global\@ignoretrue}

\begin{document}

\title{Network homophily via tail inequalities.} 


\author{Nicola Apollonio\footnote{Istituto per le Applicazioni del
		Calcolo, ``Mauro Picone'', Consiglio Nazionale delle Ricerche, Via dei Taurini 19, 00185 Roma, Italy.
		\texttt{nicola.apollonio@cnr.it}.}
	\and 
	Paolo G. Franciosa\footnote{Dipartimento di Scienze Statistiche, Universit\`a di Roma ``La Sapienza'', Piazzale Aldo Moro 5, 00185 Roma, Italy. \texttt{paolo.franciosa@uniroma1.it}.}
	\and 
	Daniele Santoni\footnote{Istituto di Analisi dei Sistemi ed Informatica ``Antonio Ruberti'', Consiglio Nazionale delle Ricerche, Via dei Taurini 19, 00185 - Rome, Italy \texttt{daniele.santoni@iasi.cnr.it}.}
	}

\date{}

\maketitle	



\begin{abstract}
Homophily is the principle whereby \caporali{similarity breeds connections}. We give a quantitative formulation of this principle within networks. Given a network and a labeled partition of its vertices, the vector indexed by each class of the partition, whose entries are the number of edges of the subgraphs induced by the corresponding classes, is viewed as the observed outcome of the random vector described by picking labeled partitions at random among labeled partitions whose classes have the same cardinalities as the given one. This is the recently introduced \emph{random coloring model} for network homophily. In this perspective, the value of any homophily score $\Theta$, namely a non decreasing real valued function in the sizes of subgraphs induced by the classes of the partition, evaluated at the observed outcome, can be thought of as the observed value of a random variable. Consequently, according to the score $\Theta$, the input network is homophillic at the significance level $\alpha$ whenever the one-sided tail probability of observing a value of $\Theta$ at least as extreme as the observed one, is smaller than $\alpha$. Since, as we show, even approximating $\alpha$ is an NP-hard problem, we resort to classical tails inequality to bound $\alpha$ from above. These upper bounds, obtained by specializing $\Theta$, yield a class of quantifiers of network homophily. Computing the upper bounds requires the knowledge of the covariance matrix of the random vector which was not previously known within the random coloring model. In this paper we close this gap. Interestingly, the matrix depends on the input partition only through the cardinalities of its classes and depends on the network only through its degrees. Furthermore all the covariances have the same sign and this sign is a graph invariant. Plugging this structure into the bounds yields a meaningful, easy to compute class of indices for measuring network homophily. As demonstrated in real-world network applications, these indices are effective, reliable, and lead to new discoveries that could not be captured by the current state of the art.
\end{abstract}
\mybreak



\section{Introduction}\label{sec:intro}
Network homophily is the phenomenon, first observed by social scientists, whereby people who share the same characteristics tend to be connected together \cite{mcp}. More abstractly, given a simple undirected graph $G$ (modeling interactions) and a partition $\pt$ of its vertices (modeling a classification based, for instance, on certain characteristics), $G$ is said to be homophilic with respect to $\pt$ if the subgraphs induced by the classes in the partition (i.e., people who presumably share the same characteristics) are significantly dense (thus showing a tendency to connect). What makes this definition vague is the inherent difficulty of explaining what ``significantly dense'' means. The main aim of this paper is to give a precise quantitative meaning to the latter term. A well-established descriptive approach to quantifying network homophily consists in defining a real valued non decreasing function $\Theta$ of the edge densities (or simply of the sizes) of the subgraphs induced by the classes of the given partition $\pt$, and to locate the value of $\Theta$ evaluated at the input on a universal scale. The function $\Theta$ is called \emph{homophily score} or \emph{descriptive homophily index}. One of the simplest descriptive indices is the \emph{homophily ratio} \cite{Zhu_et_al} (called \emph{inside edge fraction} in \cite{chakra}) which is simply defined as the fraction of the \emph{homophillic edges}, namely those edges that are induced by the classes of the partition. Also, by dividing the sum of the edge densities of the subgraphs induced by the classes by the number of classes yields another index in the scale $[0,1]$. Often, the role of the universal scale is replaced by numerical experiments on benchmark instances while homophily indices are taken from a consolidated catalog. We refer the reader to \cite{yangleskove} and \cite{lancichinetti2010} for (among other things) an annotated bibliography on both aspects.  
\mybreak
In this paper we take an inferential statistics approach, namely, we think of the  partitioned input network $(G,\pt)$ as the observed value of a random object in a probability space which induces a probability distribution on its measurable functions and, in particular, on homophily indices. The latter ones are now random variables (statistics) of which we observe a particular realization. Depending on our ability to compute their distributions, a task which is very far from obvious (see next section and Section \ref{sec:applications}), we can perform hypothesis testing, locate confidence intervals or, less ambitiously, simply compare the observed values of these statistics with their expected values under the null model represented by the chosen probability space for the random object $(G,\pt)$.
 We can think of the celebrated \emph{network modularity index} by Newman \cite{newman_m} as probably the most notable instance of this approach. The probability space for $(G,\pt)$ consists of all multigraphs on the set of vertices of $G$ with the same degrees as $G$ sampled according to Bollob\'{a}s' configuration model \cite{bollo} (see next section). The statistic in this case is the random variable 
 $$
 	\frac{1}{m}\sum_i^s\left(M_i-\Ea{M_i}\right) 
$$
where $m$ is the size of $G$, $M_i$ is the random number of edges induced by the $i$-th class of the given partition $\pt$ of $V(G)$ into $s$ classes, and $\Ea{M_i}$ is its expected value under the configuration model. The modularity of $G$ with respect to $\pt$ is simply the value of this statistic in the observed data $(m_1,\ldots,m_s)$ where $m_i$ is the number of observed homophilic edges induced by the vertices of class $i$, $i=1,\ldots,s$. Since network modularity is a non decreasing monotone function of the $m_i$'s, besides its statistical interpretation, it directly quantifies homophily in the scale $[-1/2,1]$. Actually, the range of the index could be a proper sub-interval of $[-1/2,1]$ whose extreme points correspond to the perfectly \emph{anti-homophilic} regime and to the perfectly homophilic regime, respectively. In this sense the scale $[-1/2,1]$ is not universal. However, finding the maximum of modularity over the set of labeled partitions whose classes have prescribed cardinalities, and hence the width of the range of the index over this set, is an NP-complete problem, and the same is true for the homophily ratio (see theorem \ref{thm:nphardmod}).
\mybreak
In this paper, we take a complementary approach: we think of $G$ as given while we think of $\pt$ as a sample from the probability spaces consisting of all labeled partitions of $V(G)$ having the same structure of $\pt$, equipped with the uniform distribution (see \cite{GatesetAnn} and \cite{Hoffmanetal} for alternative probability measures on $\pt$). This is the random coloring model  introduced in \cite{apollonio2022}. The random coloring model is simple enough to derive exact closed formulas for the first two moments of the random variables $M_i$'s which are now defined on the probability space of the labeled partition of the same structure as $\pt$ (see Section \ref{sec:homophily_in_random_models}). On these grounds, we propose a multidimensional extension of the well known three-sigma rule of thumb by providing a Cantelli-type upper bound for the probability of the tails $\left[\Theta(M_1,\ldots,M_s)\geq \Theta(m_1,\ldots,m_s)\right]$ and $\left[\Theta(M_1,\ldots,M_s)\leq \Theta(m_1,\ldots,m_s)\right]$ where $\Theta$ is the homophily score. These bounds are coupled together and yield an index that ranges in $[-1,1]$ which we take as a quantifier of network homophily. Moreover, its absolute value, can be interpreted as an estimated significance of the deviation of the observed homophily from the expected one under the random coloring model. To illustrate the idea, consider the following extremely simplified problem: suppose we want to asses whether a set $C$ of $c$ vertices of a graph $G$ induces a subgraph $H$ which is significantly denser among the subgraphs induced by the same number of vertices; we may think of the edge-density $y$ of $H$ as a sample of the random variable $Y$ described by the density of the subgraph induced by uniform sampling of sets of $c$ vertices from $G$; by the results in \cite{apollonio2022}, we known the expectation $\am{Y}$ and the variance $\sigma^2$ of $Y$; hence we can compute the observed $z$-score defined as $z=(y-\am{Y})/\sigma$; the denser is $H$ the higher is $z$. Well, but how much high $z$ is? How can we assess that what we observe is statistically significant? If we knew the tail probability 
$$\pra{\frac{Y-\Ea{Y}}{\sigma}\geq z}$$
then we could immediately answer: if $y\leq \am{Y}$, then $H$ is not significantly dense at all, while the smaller is the tail probability above, the higher is the significance of the deviation from the expected value. However, already in this simplified problem, the probability above is even hard to approximate (see Section \ref{sec:homophily_in_random_coloring_models}). Here comes into our aid a kind of three-sigma rule of thumb: let 
$$Z=\frac{Y-\Ea{Y}}{\sigma}$$
be the $Z$-score of $Y$, namely the standardized version of $Y$; by Cantelli's inequalities (\eqref{eq:cantelli_r} and \eqref{eq:cantelli_l} in Section \ref{sec:homophily_in_random_coloring_models}) the inequality        
$$\pra{Z\geq z}\leq \frac{1}{1+z^2}$$
holds and tells us that, for $z\geq 0$, the function $z\mapsto \frac{z^2}{1+z^2}=1-\frac{1}{1+z^2}$ is not decreasing in the range $[0,1]$ and that the higher $z$, the higher its significance. By applying the same inequality to $-Z$ we obtain a symmetric bound for the hypo-dense behavior, namely the significance of very large deviation of $Y$ from $\am{Y}$ in the opposite direction; similarly, for $z\leq 0$, the function $z\mapsto -\frac{z^2}{1+z^2}$ is not decreasing in the range $[-1,0]$ and that the higher $|z|$ is, the higher is its significance. Coupling the bounds yields an index in the scale $[-1,1]$ which ranges between the hypo-dense and hyper-dense case. The multidimensional extension of this idea is technically more difficult in that the bound crucially depends on the covariance matrix of the random vector $M$, which was not previously known within the random coloring model. In this paper we close this gap by computing the covariance matrix of subgraph sizes under the random coloring model. Interestingly, the matrix depends on the input partition only through its profile and on the network only through its degree sequence. Furthermore, all the covariances have the same sign and this sign is a graph invariant. 
\mybreak
The knowledge of the covariance structure of the random outcome $M$ allow us to define a class of quantifiers of network homophily by specializing the function $\Theta$. Although this leads to a variety of possibilities to measure network homophily (see Section \ref{sec:H}), what we deem more interesting of our work is the methodological approach which consists in viewing the input partitioned network, as a random object where randomness is inherited from random partitions rather than from random graphs. In principle, this is nothing more than a reasonable modeling choice. However it has some important implications:
\begin{itemize}
	\item[--] we provide an universal scale for our indices; in fact their range is always (and tightly) the interval $[-1,1]$, unlike the ranges of other indices, such as the modularity and homophily ratio indices, whose maxima are difficult to locate, the corresponding optimization problems being typically NP-complete (see Theorem \ref{thm:nphardmod});
	\item[--] due to the simplicity of the random coloring model we gain the possibility of sharper analysis with respect to random graph models; for example, we have complete knowledge of the covariance structure of $M$ which can be exploited to estimate the significance of the observed homophily score via tail inequalities and hence to take this estimated significance as a homophily quantifier.
\end{itemize}
We discuss and substantiate these claims in Section \ref{sec:applications} where we show, for instance, as expected, that the inherent homophilic nature of the Homo sapiens PPI network with respect to the location of corresponding genes on chromosomes is captured neither by the descriptive indices like the homophily ratio nor by Newman's modularity, while it is captured by ours.

The rest of the paper goes as follows. In Section \ref{sec:homophily_in_random_models} the general concept of homophily in random models is formally described introducing the ground of our approach to homophily: the random coloring model; in Section \ref{sec:covariance} we elicit the second-order moment structure of the random outcome $M$ and study its properties including its computational complexity; this results are next exploited in Section \ref{sec:H} where we propose a class of quantifiers of homophily and show that each of them can be efficiently computed; finally in Section \ref{sec:applications} we substantiate our approach through examples and real-world applications. Section \ref{sec:concl} closes the paper with concluding remarks.

\section{Homophily in random models.}\label{sec:homophily_in_random_models} 
Throughout the rest of the paper we use the following standard notation: for a positive integer $k$ the symbol $[k]$ stands for the set $\{1,\ldots,k\}$ of the first $k$ positive integers. The edge-set of a graph $G$ is denoted by $E(G)$. Symbols $K_n$, $P_n$, $K_{1,n}$ and $K_2^{(t)}$ denote the isomorphism classes of the complete graph on $n$ vertices, the path on $n$ vertices, the tree with $n+1$ vertices and $n$ leaves, and the graph consisting of $t$ disjoint edges. The number of copies of $P_3$ in $G$ is denoted by $\pi_3(G)$. Let $G$ be a graph on $[n]$ with $m$ edges. Let $\pt$ be a partition of $[n]$ into $s$ distinguishable classes. Let $c_i$ be the number of vertices in class $i$, $i\in[s]$. The $s$-tuple $\col=(c_1,\ldots,c_s)$ is the \emph{profile} of $\pt$. Since labeled partitions $\mathcal{P}$ of $[n]$ of profile $\col$ correspond bijectively to surjective maps $f:V(G)\rightarrow [s]$ where $f^{-1}(i)=c_i$, for $i\in [s]:=\{1,2,\ldots,s\}$, we can equivalently think of the partitioned network $(G,\mathcal{P})$ as the colored network $(G,f)$ where for each vertex of $G$, $f$ tell us the color of that vertex. Accordingly, we rephrase the definition of homophily by saying \emph{homophillic with respect to $f$} rather than \emph{homophillic with respect to $\mathcal{P}$} and we refer to $\col$ as \emph{the profile of coloring $f$}. Also, we say that an edge of $G$ is $i$-homophilic (or homophilic of color $i$) if both endpoints have color $i$. Let $(m_1,\ldots,m_s)'$---here and henceforth $u'$ denotes the transpose of vector $u\in\mathbb{R}^d$, for some natural number $d$, usually thought of as column vector--- be the observed outcome for the input $(G,\pt)$, where $m_i$ is the number of $i$-homophilic edges under $f$.
\mybreak
There are mainly three options to prescribe randomness for $(G,f)$:
\begin{enumerate}[label=(\alph*)]\itemsep0em
	\item\label{com:a} we may think both $G$ and $f$ as random objects, namely, $\mathcal{G}_n$ is a set of graphs on $[n]$, $\Pi_\col$ is the set of coloring of $[n]$ of profile $\col$, and $\pd_{\mathcal{G}_n,\col}$ is a probability measure on $\mathcal{G}_n\times \Pi_\col$;
	\item\label{com:b} we may think of $G$ as given while of the coloring $f$ as a sample from $\Pi_\col$ according to a probability measure $\pd_\col$ on $\Pi_\col$ ($\pd_\col$ can therefore be viewed as conditional probability); 
	\item\label{com:c} we may think of $f$ as given while the graph $G$ as as sample from the random graph $\mathcal{G}_n$ according to a probability measure $\pd_{\mathcal{G}_n}$ on $\mathcal{G}_n$ ($\pd_{\mathcal{G}_n}$ can therefore be viewed as conditional probability);  
\end{enumerate} 
Since models in \ref{com:a} can be recovered by the knowledge of models in \ref{com:b} and \ref{com:c} via the factorizations 
$$\pd_{\mathcal{G}_n,\col}(G,f)=\pd_{\mathcal{G}_n,\col}(G,f|G)\pd_{\mathcal{G}_n}(G)$$
$$\pd_{\mathcal{G}_n,\col}(G,f)=\pd_{\mathcal{G}_n,\col}(G,f|f)\pd_\col(f),$$
where we identify the marginals of $\pd_{\mathcal{G}_n,\col}$ with $\pd_{\mathcal{G}_n}$ and $\pd_\col$, we may concentrate on the models in \ref{com:b} and \ref{com:c}. Whichever models of \ref{com:a}, \ref{com:b}, and \ref{com:c} applies, 
\begin{itemize}\itemsep0em
	\item[--] $M=(M_1,\ldots M_s)'$ is the \emph{outcome} in the chosen probability space (of which we observe $(m_1,\ldots,m_s)$); $M_i$ is the random variable defined as the number of $i$-homophilic edges;
	\item[--] $\am{M}=(\am{M}_1,\cdots,\am{M}_s)$ with $\am{M}_i=\Ea{M_i}$, $i\in [s]$ and $\Ea{\cdot}$ being the \emph{expectation} under the chosen model, is the the \emph{expected} or \emph{most typical} outcome even if, despite of being typical, it might not belong to the support of $M$ (for instance, because its coordinates might not be integer numbers).  
	\item[--] $\sigma_i^2=\var{M_i}=\Ea{(M_i-\am{M}_i)^2}$ is the variance of $M_i$, $\Sigma=\Ea{(M-\am{M})(M-\am{M})'}$ Is the covariance matrix of $M$ and, provided that $\sigma_i>0$ for all $i\in [s]$, $\Gamma=D^{-\frac{1}{2}}\Sigma D^{\frac{1}{2}}$ is the correlation matrix of $M$, where $D$ is the diagonal matrix with diagonal entries $\sigma^2_1,\ldots,\sigma^2_s$; 
	\item[--] $Z=(Z_1,\ldots,Z_s)$ and $z=(z_1,\ldots,z_s)$ are the $z$- and the observed $z$-scores vectors, where, for $i\in [s]$, $Z_i=(M_i-\am{M_i})/\sigma_i$ and $z_i=(m_i-\am{M_i})/\sigma_i$; notice that the covariance matrix of $Z$ is $\Gamma$.
\end{itemize}
\mybreak
Probably the most notable instance of the model in \ref{com:c} is the celebrated \emph{network modularity index} by Newman \cite{newman_m}---more generally, Newman studies from a quantitative viewpoint the broader phenomenon of \emph{assortative mixing} of which network homophily can be seen as categorical manifestation \cite{newman}---. Indeed the modularity of $G$ with respect to $f$ in the sense of \cite{newman_m} is the index
\begin{equation}\label{eq:utile}
	\mathfrak{q}(f)=\sum_i^s\left(\frac{m_i}{m}-\left(\frac{D_i}{2m}\right)^2\right)
\end{equation}
where $D_i$ is the sum of the degrees of the vertices in the $i$-th color class. To see why the index above fits the model in \ref{com:c} let $M_i$ be the random number of $i$-homophilic edges in a (multi)graph $G$ sampled according to Boll\'{o}bas' configuration model \cite{bollo}---a random (multi)graph with the same degrees as $G$---and observe \cite{newman_m} that $\left(\frac{D_i}{2m}\right)^2$ is the expectation of $M_i/m$ under this model. Therefore $\mathfrak{q}(f)$ takes the form 
\begin{equation}\label{eq:modularity}
\frac{1}{m}\sum_i^s\left(m_i-\Ea{M_i}\right) 
\end{equation}
which expresses $\mathfrak{q}(f)$ as an appropriate multiple of the sum of the deviations of the observed homophillic edges in class $i$ from their expected value under a particular null model (the configuration model in the present case). The reason why one resorts to probability spaces of multigraphs rather than simple graphs is because, unlike the Erd\H{o}s-R\'{e}nyi model $G(n,p)$, the edges of a simple random graph $\mathcal{G}(d)$ on $[n]$ with prescribed degree sequences $d$, are correlated. In fact, even estimating the probability of a single edge can be challenging and that makes analysing $\mathcal{G}(d)$ difficult and more difficult than the case of multigraphs (\cite{GaoO}). However, in principle, the choice of any null model $(\mathcal{G}_n,\pd_{\mathcal{G}_n})$ yields a modularity-like index having the same general structure \eqref{eq:modularity}.
\subsection{The random Coloring Model}\label{sec:homophily_in_random_coloring_models} 
The models in \ref{com:b} are in some sense complementary to those in \ref{com:c}. Choosing for $\pd_\col$ the uniform measure we obtain the recently introduced \emph{random coloring model} for network homophily \cite{apollonio2022} based on implicit ideas in \cite{park2007}. The probability space consists thus of the set  $\Phi_\col$ of all colorings of $G$ of profile $\col$ equipped with the uniform measure
\begin{equation}\label{eq:unif_col}
	\pd_\col=c_1!\cdots c_s!/n!.
\end{equation}
This is the model we adopt in this paper. The random coloring model is simple enough to allow exact closed formula for the first two moments of the random variables $M_i$'s (see \cite{apollonio2022}). In fact, within this model there is a handy representation of the random variables $M_i$'s by means of the Bernoulli random variables $X_i^e$, $e\in E(G)$, $i\in [s]$, defined as the indicators of the event: \emph{edge $e$ is $i$-homophilic}. Indeed,
$$M_i=\sum_{e\in E(G)}X_i^e,$$
so that we have the following expressions for the mean and the variance of $M_i$, $i=1,\ldots,s$ (\cite{apollonio2022})
\begin{align}
	\overline{M_i}&=m\frac{\kn{c_i}{2}}{\kn{n}{2}},\label{eq:media}\\
	\sigma^2_i&=m\frac{\kn{c_i}{2}}{\kn{n}{2}}\left(1-m\frac{\kn{c_i}{2}}{\kn{n}{2}}\right)+2\left\{\left(\frac{\kn{c_i}{3}}{\kn{n}{3}}-\frac{\kn{c_i}{4}}{\kn{n}{4}}\right)\pi_3(G)+\frac{\kn{c_i}{4}}{\kn{n}{4}}\left(m \atop 2\right)\right\}\label{eq:varianza}
\end{align}
where, following the notation in \cite{knu}, for a positive integer $a$ and a nonnegative integer $q$, we  denote by the symbol $\kn{a}{q}$ the \emph{$q$-th falling factorial} of $a$, namely $\kn{a}{q}=a(a-1)\cdots(a-q+1)$, with $\kn{a}{0}=1$. Thus, $\kn{a}{2}=a(a-1)$.
After setting $\kn{\col}{2}=(\kn{c_1}{2},\ldots,\kn{c_1}{2})$, the most typical outcome in the random coloring model with parameters $\col$ and $n$ reads as 
$$\overline{M}=\frac{m}{\kn{n}{2}}\kn{\col}{2},$$
namely, the most typical outcome corresponds to a (possibly unattainable) coloring having the property that each color class induces a subgraph having the same density as input graph. In \cite{apollonio2022} it is shown that both the marginal mean $\am{M_i}$ and the marginal variance $\sigma_i^2$ can be computed in optimal $O(s+m)$ worst-case time. It is instructive to compute this vector for the complete graph on $n$ vertices and a given profile $\col$. It is clear that, for any coloring $g\in \Phi_{\col}$, the number of edges in the subgraph induced by color $i$, is exactly $\kn{c_i}{2}/2$. Hence $M$ is constant on $\Phi_{\col}$ and $\overline{M}=\kn{\col}{2}/2$. On the other hand, the density of the complete graph is 1 while the coefficient of $\kn{\col}{2}$ in the representation of $\overline{M}$ is half this density. Therefore, by applying the formula, we recover the same result. Using the corresponding formula, the reader may check that $\sigma^2_i=0$, for all $i\in [s]$. Notice that $\pi_3(K_n)=\kn{n}{3}/2$. 
\paragraph{Homophily in the random coloring model.}
The null hypothesis behind the random coloring model is the absence of homophily, namely the coloring $f$ does not correlate with $G$. This means that the effect of $f$ on $G$, measured by some monotone function $\Theta$ of the number of homophilic edges, should be statistically the same as the effect of any other (equally likely) coloring of $G$ with the same profile as $f$. We aim at devising an index, in the form of a non decreasing function of the number of homophilic edges, such that 
\begin{enumerate}[label={\rm (\roman*)}]\itemsep0em
\item\label{com:1}on the one hand it locates the input $(G,f)$ in the universal scale $R$ (i.e., independent on the particular input graph)  whose extreme values correspond to the anti-homophilic and perfect homophilic behavior, respectively. Note that locating the range of the modularity  value or homophily ratio index is a far from  obvious problem (see Theorem \ref{thm:nphardmod}), since it requires locating their maxima over the labeled partitions of the given profile,
\item\label{com:2}and on the other hand, values of the index are measures of the significance of the deviation of the number of homophilic edges from their expected value under the random coloring model.  
\end{enumerate}   
Suppose for the moment that we know the distribution $M$ induced by $\pd_\col$. Consider the arithmetic mean of the $z$ scores 
$$A:=A(Z_1,\ldots,Z_s)=\frac{1}{s}\sum_i^sZ_i.$$ 
This amounts to choose the statistic $\Theta$ as a function of $(M_1,\ldots,M_s)$ as
$$\Theta:\R^s\rightarrow \R,\quad x\mapsto \frac{1}{s}\sum_i^s \frac{x_i-\am{M}_i}{\sigma_i}$$
which is trivially an affine non decreasing function of its arguments. The index
\begin{equation}\label{eq:index_a}
	\mathfrak{a}(G,f)=\left\{
	\begin{array}{rr}
		1-\pra{A(Z)\geq A(z)} & \text{if $A(z)\geq 0$}\\
		\pra{A(Z)\leq A(z)}-1& \text{if $A(z)< 0$}\\
	\end{array}
	\right.
\end{equation}  
ranges in $[-1,1]$ and meets the requirements in \eqref{com:1} and \eqref{com:2}.
Unfortunately, knowing the distribution of $M$ is a very unlikely case. Indeed, a straightforward reduction to CLIQUE shows that given a certain positive integer $\kappa$ even deciding whether the marginal tail probability $\pra{M_1\geq \kappa}$ is positive or not is an NP-complete problem: just choose for $\kappa$ the number $c_1(c_1-1)/2$ where $c_1$ is the number of vertices of color $1$: $G$ has a clique on $c_1$ vertices if and only if $\pra{M_1\geq \kappa}>0$. Therefore, since CLIQUE is not even in the APX class, it follows that the probabilities in \eqref{eq:index_a} are even hard to approximate. While this circumstance is very frustrating, not being able to calculate the probability of a given event exactly, is a rather common event in Probability and Statistics. The usual remedy is to resort to so-called tail inequalities. The typical couple of one-sided tail inequalities for the random variable $\Psi(Z)$, where $\Psi: \R^s\rightarrow \R$, read as
$$
\left\{
\begin{array}{rr}
	\pra{\Psi(Z)\geq \Psi(z)}\leq B(\Psi(z)) & \text{if $\Psi(z)>0$}\\
	\pra{\Psi(Z)\leq \Psi(z)}\leq B(\Psi(z)) & \text{if $\Psi(z)<0$}\\
\end{array}
\right.
$$ 
where $B$ is a non non-negative real valued function of its argument which depends in general on the mixed moments of the random vector $Z=(Z_1,\ldots,Z_s)'$. For our purposes, the most interesting concentration inequality is the celebrated Cantelli's inequality. This inequality requires the knowledge of covariance matrix of $M$ which was not previously known and for which we give a closed form expression (see Theorem \ref{thm:inv_sigma}) in the next section which depends on $G$ only through its degrees and on $f$ only through its profile. Let us illustrate how we can mimicking index $\mathfrak{a}$ via tail inequalities, assuming the knowledge of the covariance structure of $M$. Consider classical uni-variate Cantelli's inequalities for a random variable $X$ with finite expectation and finite variance $\sigma^2$. For a positive real number $t$, these inequalities are 
\begin{align}
	\pra{X-\Ea{X}\geq t}\leq \frac{\sigma^2}{t^2+\sigma^2}\label{eq:cantelli_r}\\
	\pra{X-\Ea{X}\leq -t}\leq \frac{\sigma^2}{t^2+\sigma^2}\label{eq:cantelli_l}
\end{align}
Notice, the symmetry in the bounds which implies that both the left and the right tail are approximated in the same way. The more symmetric the distribution of $X-\Ea{X}$ is, the more effective these bounds are.  Since $A(Z)$ is a centered univariate random variable, with finite variance
$$\var{A(Z)}=\Ea{\left[\frac{\mathbf{1}'Z}{s}\right]^2}=\Ea{\frac{\mathbf{1}'ZZ'\mathbf{1}}{s^2}}=\frac{\mathbf{1}'\Ea{ZZ'}\mathbf{1}}{s^2}=\frac{\mathbf{1}'\Gamma\mathbf{1}}{s^2}$$ where $\Gamma$ is the covariance matrix of $Z$ and $\mathbf{1}$ is the $s$-dimensional all ones vector, the bounds above apply. Therefore, the index
\begin{equation}\label{eq:hgfa}
	a(G,f)=\text{sgn}(A(z))\frac{(sA(z))^2}{(sA(z)^2)+\mathbf{1}'\Gamma\mathbf{1}}
\end{equation}
where $z$ is the observed $z$-score and $\text{sgn}$ is the signum function, ranges in $[-1,1]$, meets the requirements in \eqref{com:1} and (approximately) the requirement in \eqref{com:2}. In view of the hardness result, this is the best we can hope for by pursuing this approach.

\section{Covariance structure of subgraph densities}\label{sec:covariance}
We saw in the previous section that the covariance matrix of the outcome of the random coloring model, crucially enters the upper bounds on the probability of deviating from the most typical outcome and finally in the estimates of $\mathfrak{a}$. This section is devoted to the study of the covariance structure of $M$ and this completes the knowledge about the second-order moments of $M$ started in \cite{apollonio2022}. 

\begin{theorem}\label{thm:gamma}
	Let $(G,f)$ be the input graph in the random coloring model. Let $G$ have $n$ vertices and $m$ edges and $f$ have profile $\col$. Then, for all $i,\,j$, $i\not=j$
	$${\rm cov}(M_i,M_j)=\kn{c_i}{2}\kn{c_j}{2}\left[\frac{2}{\kn{n}{4}}\left\{{m \choose 2}-\pi_3(G)\right\}-\left(\frac{m}{\kn{n}{2}}\right)^2\right],$$
	which can be written as 
	\begin{equation}\label{eq:scalar_cov}
		{\rm cov}(M_i,M_j)=\gamma(G)\kn{c_i}{2}\kn{c_j}{2}	
	\end{equation}
	where
	\begin{equation}\label{eq:gamma}
		\gamma(G)=\frac{2}{\kn{n}{4}}\left\{{m \choose 2}-\pi_3(G)\right\}-\left(\frac{m}{\kn{n}{2}}\right)^2
	\end{equation}
	depends only on $G$ through its density and the distribution of its degrees.
\end{theorem}
\begin{proof}
	Let $i\in[ s]$ be any color. Edge $e\in E(G)$ is $i$-monochromatic under a coloring $g$, if both end-vertices of $e$ have color $i$. Let $X^e_i$ be the Bernoulli random variable defined as the indicator of the event: \emph{edge $e$ is $i$-monochromatic}. Recall from Section \ref{sec:homophily_in_random_coloring_models}  that for every color $i$, $M_i$ has the representation $M_i=\sum_{e\in E(G)}X^e_i$.  Let now $i,\,j\in [s]$ be two distinct colors. By definition ${\rm cov}(M_i,M_j)=\Ea{M_iM_j}-\Ea{M_i}\Ea{M_j}$ and by \eqref{eq:media} we already know $\Ea{M_i}\Ea{M_j}$. The rest of the proof follows by expanding $\Ea{M_iM_j}$:
	$$\Ea{M_iM_j}=\Ea{\left(\sum_{e\in E(G)}X^e_i\right)\left(\sum_{e'\in E(G)}X^{e'}_j\right)}.$$
	By expanding the product inside the expectation, we obtain the polynomial expression
	$$\sum_{(e,e')\in E(G)\times E(G)}X^e_iX^{e'}_j.$$
	Since no $i$-monochromatic edge $e$ can share a vertex with a $j$-monochromatic edge $e'$, we have the relations 
	$$X^e_iX^{e'}_j=0 \quad \forall (e,e')\in E(G)\times E(G)\quad\text{such that}\quad e\sim e'$$  
	where, for edges $e$ and $e'$ (with possibly $e=e'$) we write $e\sim e'$ if $e$ and $e'$ have at least a vertex in common. Therefore, by linearity 
	$$\Ea{M_iM_j}=\sum_{e\not\sim e'}\Ea{X^e_iX^{e'}_j}.$$
	Now $\Ea{X^e_iX^{e'}_j}=\pra{X^e_i=1, X^{e'}_j=1}$ because the variables  involved are Bernoulli and, if $e\not\sim e'$, then 
	$$\pra{X^e_i=1, X^{e'}_j=1}=\frac{(n-4)!}{(c_i-2)!(c_j-2)!}\left(\frac{n!}{c_i!c_j!}\right)^{-1}=\frac{\kn{c_i}{2}\kn{c_j}{2}}{\kn{n}{4}}.$$
	Since $\pra{X^e_i=1, X^{e'}_j=1}$ does not depend on the pair $(e,e')$, it follows that 
	$\Ea{M_iM_j}$ is $\frac{\kn{c_i}{2}\kn{c_j}{2}}{\kn{n}{4}}$ times the number of ordered pairs $(e,e')$ with no vertex in common and this number is twice ($=2!$) the difference between the number ${m \choose 2}$ of unordered pairs of edges and the number of pairs sharing exactly one vertex. Since the latter number is precisely the number $\pi_3(G)$ of paths of $G$ with two edges we obtain  
	$$\Ea{M_iM_j}=2\frac{\kn{c_i}{2}\kn{c_j}{2}}{\kn{n}{4}}\left\{{m \choose 2}-\pi_3(G)\right\}.$$
	In view of \eqref{eq:varianza}, after subtracting $\Ea{M_i}\Ea{M_j}$ from the expression above we obtain the required expression. Now \eqref{eq:scalar_cov} and \eqref{eq:gamma} follow by trivial manipulations.
\end{proof}

Notice that for $i=j$, the expression above does not specialize to the variance given in \eqref{eq:varianza}. However, what is remarkable is that the covariances have the same sign and this sign depends only on $G$ and not on the probability space. Therefore the sign of $\gamma(G)$ is a graph invariant. Interestingly we have:
\begin{theorem}\label{thm:inv_sigma}
	Let $\mathcal{G}$ be a class of graph whose members have the same number of edges and vertices and hence the same edge-density. Then $\gamma(G)$ is maximized on $\mathcal{G}$ by the graphs having minimum possible sum of squares of the degrees, and it is minimized on $\mathcal{G}$ by those graphs having maximum possible sum of squares of the degrees.  
\end{theorem}\label{coro:stars_paths}
\begin{proof} By \eqref{eq:gamma}, $\gamma(G)$ is maximized (minimized) over $\mathcal{G}$, by the graphs $G$ for which which $\pi_3(G)$ is minimum (maximum). Let $d_G=(d_1,\ldots,d_n)$ be the degree sequence of a graph $G$ and observe that $2\pi_3(G)=\sum_id_i^2+2m$. Therefore $G^*$ maximizes (minimizes) $\gamma(G)$ over $\mathcal{G}$ if and only if $d_{G^*}$ minimizes (maximizes) $\sum_id_i^2$ over the same set. 
\end{proof}
As an example, let us consider the case of trees. Since among trees $\sum_id_i^2$ is maximized by stars, while trees that minimize $\sum_id_i^2$ are paths, we have explicitly:
$$\gamma(K_{1,n})=-\frac{1}{\left(n+1\right)^2},\qquad\text{and}\qquad \gamma(P_n)=\frac{1}{n^2(n-1)}.$$ 
We can also express $\gamma(G)$ in terms of the first two moments of the degree distribution of $G$. Let $\delta_1$ and $\delta_2$ be the first two moments of the degree distribution of $G$---recall that degree distribution of $G$ is the function which associates each integer $t\in\{0,1,\ldots,n-1\}$ with the fraction of vertices having $t$ neighbors in $G$---. One has $\delta_1=2m/n$ and $\delta_2$ is the sum of the squares of the degrees divided by $n$. After simple but tedious manipulations it follows that
$$\gamma(G)=\frac{n}{\kn{n}{4}}\left(\frac{2n-3}{2n-2}\delta_1^2+\frac{1}{2}\delta_1-\delta_2\right)$$
therefore, 
$$ \gamma(G)\leq 0 \Longleftrightarrow \frac{\delta_2-\delta_1^2}{\delta_1}\geq \frac{1}{2}\left(1-\frac{\delta_1}{n-1}\right)$$
After setting $\upsilon(G)=\frac{\delta_2-\delta_1^2}{\delta_1}$ and denoting by $\rho(G)$ the edge density of $G$, the relation above reads as
$$ \gamma(G)\leq 0 \Longleftrightarrow \upsilon(G)\geq \frac{1-\rho(G)}{2}$$
and relates the sign of the correlation of the marginal of $M$ (or of $Z$) to the so-called \emph{index of dispersion} of the degree distribution of $G$. This index is defined as the ratio between the variance and the mean of a given distribution. The index of dispersion of the degree distribution of $G$ is thus precisely $\upsilon(G)$. Since for any graph $G$, it holds that $0\leq\frac{1-\rho(G)}{2}< 1$, it follows that values of the index larger than one qualify over-dispersed distributions while values of the index smaller than one qualify under-dispersed distributions. Hence networks with over-dispersed degree distribution yield negatively correlated marginals while regular graphs yield  positively correlated marginals.
\mybreak
Besides its interesting combinatorial meaning, invariant $\gamma(G)$ implies a strong linear algebraic structure for the covariance matrix of $M$ and this structure has striking algorithm consequences for computing Cantelli's bounds our homophily quantifiers (to be defined in the next section) rely on. Recall that a \emph{$Z$-matrix} is a square matrix $A$ over the real field whose off-diagonal elements are nonpositive and that an \emph{$M$-matrix} is a nonsingular $Z$-matrix  whose real eigenvalues are positive \cite{fiedler}. Clearly every symmetric positive definite $Z$-matrix is an $M$-matrix. $M$-matrices are characterized in several different ways (see \cite{fiedler} and \cite{horn}). For our purposes the most interesting such characterization is the following.
\begin{theorem}[\cite{fiedler}]
	A $Z$-matrix is an $M$-matrix if and only if it has a nonnegative inverse. 
\end{theorem} 
The following theorem essentially recognizes the rich structure of the covariance matrix of $M$ while the methodological implications are given in the next section. 
\begin{theorem}\label{thm:cov_mat}
	Let $(G,f)$ be the input pair in the random coloring model with $G$ and $f$ as in Theorem \ref{thm:gamma}. 
	\begin{enumerate}[label={\rm (\roman*)}]\itemsep0em
		\item\label{com:i} The covariance matrix $\Sigma$ of the outcome $M$ of the random coloring model $(\Phi_{\col},\mathbb{P}_{\col})$ is of the form
		$$\Sigma=Q+\gamma(G)\kn{\col}{2}{\kn{\col}{2}}',$$	
		where $Q=\diagonal{q_1,\ldots,q_s}$ with  
		$$q_i=\sigma^2_i-\gamma(G)\left(\kn{c_i}{2}\right)^2$$ 
		where $\sigma^2_i$ is given in \eqref{eq:varianza} and $\gamma(G)$ in Theorem \ref{thm:gamma}. Hence $\Sigma$ is a rank 1 update of the diagonal matrix $Q$. If $Q$ is invertible, so is $\Sigma$ and by the Sherman-Morrison formula, $\Sigma^{-1}$ is of the form
		$$\Sigma^{-1}=\diagonal{q_1^{-1},\ldots,q_s^{-1}}-\frac{\gamma(G)}{1+\gamma(G)a'Qa}aa',$$
		where $a=Q^{-1}\kn{\col}{2}$.
		\item\label{com:ii} If $\gamma(G)\leq 0$, then $\Sigma_S$ is an $M$-matrix for all principal minors $\Sigma_S$ with $S\not=\emptyset$ while if $\gamma(G)\geq 0$, then  $(\Sigma^{-1})_S$ is an $M$-matrix for all principal minors $(\Sigma^{-1})_S$ with $S\not=\emptyset$.
	\end{enumerate}	
		The same conclusions in \ref{com:i} and \ref{com:ii} hold when $\Sigma$ is replaced by the correlation matrix $\Gamma$.
\end{theorem}
\begin{proof}
	Statement \ref{com:i} follows straightforwardly by rewriting equations \eqref{eq:varianza} and \eqref{eq:scalar_cov} in matrix form and then by applying the Sherman-Morrison formula \cite{gol}. As for the last statement, observe that if $\Sigma$ is positive definite, then so is $D=\diagonal{\sigma^2_1,\ldots,\sigma^2_s}$. Hence, replacing $\Sigma$ by $D^{-\frac{1}{2}}\Sigma D^{-\frac{1}{2}}$, affects neither the sign of $\gamma(G)$ nor the property of being a rank 1 update of a diagonal matrix. It remains to prove Statement \ref{com:ii}. Both $\Sigma$ and $\Sigma^{-1}$ are positive definite matrices and both are rank 1 updates of a diagonal matrix, namely, they are of the form $L+\lambda uu'$, where $L$ is a diagonal nonsingular real matrix, $\lambda$ is a real number, and $u\in \R^s_+$. Therefore, for every nonempty subset $S$ of $[s]$ it holds that $(L+\lambda uu')_S=L_S+\lambda u_Su_S'$ and $u_S>0$. Let $A_S=L_S+\lambda u_Su_S'$ It can be seen in various way (e.g., by the \emph{Cauchy interlacing eigenvalues Theorem}, \cite{horn}) that $A_S$ is a symmetric definite positive matrix for every nonempty subset $S$ of $[s]$. Therefore if $\lambda\leq 0$, then $A_S$ is a positive definite $Z$-matrix and hence $A_S$ is an $M$-matrix for every nonempty subset $S$ of $[s]$.
\end{proof}
In view of the  theorem, finding $\Sigma$, $\Gamma$ and their inverses can be accomplished very efficiently, in fact in optimal $O(s^2+m)$ worst-case time, where $s$ is usually a small constant: computing $\gamma(G)$ requires $O(m)$ and due to the structure of the covariance matrix, $O(s^2)$ is needed to write $\Sigma$ or $\Gamma$, to perform inversion and vector multiplication.
\begin{remark}\label{rem:uncorr}
Beside the sign, the absolute value of $\gamma(G)$ is also quite informative. Indeed, by \eqref{eq:gamma}, if $G$ is sparse enough, like the graphs considered in Section \ref{sec:applications}, $|\gamma(G)|$ is very small and the covariance is nearly zero. Consequently, $M$ and $Z$ have essentially uncorrelated marginals so that $\Sigma$ is very close to the diagonal matrix $D$.  
\end{remark}

\section{Measuring homophily: a class of indices for network homophily}\label{sec:H}
We have already defined the homophily index $a(G,f)$ in Section \ref{sec:homophily_in_random_coloring_models}. Since the index depends only on the observed $z$-scores and the correlation matrix of $M$, after the results of the previous section, we see that $a(G,f)$ can be computed very efficiently. The rationale behind the definition of index $a$ is that if $G$ is homophilic with respect to $f$, then this behavior should be exhibited globally: either in most classes or in a few classes but in an extremely significant way. The same applies to anti-homophily. We can therefore view the arithmetic mean $A(z)$ of the observed $z$-scores as a \emph{score} of the homophily of $G$ with respect to $f$ and $a(G,f)$ as a measure of the statistical significance of such a score. Moreover, as seen in Section \ref{sec:homophily_in_random_coloring_models}, $A(z)$ is a real valued non decreasing affine function of the observed outcome $(m_1,\ldots,m_s)'$ so that, in the same vein, we can define and efficiently compute an entire class of indices simply by choosing a different score in the form of an affine non decreasing real valued function $\Theta$. The recipe goes as follows. Let $w\in \R^s$ have non negative components and consider the linear map  $\R^s\ni x\mapsto w'x\in\R$. This map clearly restricts to a measurable function of $\Phi_\col$ so that, by linearity, the score $\Theta$ defined by $\Theta(M)=w'(M-\am{M})$, describes a centered random variable with covariance matrix 
\begin{equation}\label{eq:cov_theta}
\Sigma_\Theta=w'\Sigma w 
\end{equation}
where $\Sigma$ is the covariance matrix of $M$. Clearly $\Theta$ is a non decreasing real valued function in the coordinates of the observed outcome $(m_1,\ldots,m_s)'$. Therefore, by mimicking the construction of $a(G,f)$ and after setting $y=(m_1-\am{M_1},\ldots,m_s-\am{M_s})'$, we can associate with the score $\Theta$ the index $j_\Theta(G,f)$, varying in $[-1,1]$, defined by
\begin{equation}\label{eq:general_score}
	j_\Theta(G,f)=\text{sgn}(\Theta(y))\frac{(\Theta(y))^2}{(\Theta(y))^2+w'\Sigma w}.
\end{equation}
By construction, for any choice of the score $\Theta$ and hence of the weighting $w$, index $j_\Theta(G,f)$ meets the requirements in \eqref{com:1} and (approximately) the requirement in \eqref{com:2} of Section \ref{sec:homophily_in_random_coloring_models}. Index $j_\Theta(G,f)$ can be specialized in several reasonable ways. For instance, according to the the terminology of \cite{chakra}, typical choices for the score $\Theta$ could be based on    
\begin{itemize}
	\item[--] the \emph{edge-inside fraction}, also called \emph{homophily ratio}, defined as $m_i/m$; this amounts to choose $w=\mathbf{1}'/m$ so as to have a score in the range $[0,1]$;
		\item[--] the \emph{average internal degree} defined as $2m_i/c_i$; this amounts to choose $w$ proportional to $2(\frac{1}{c_1},\ldots,\frac{1}{c_s})'$ by a coefficient $\nu$; possible choices for $\nu$ are: the reciprocal of the maximum degree $\Delta(G)$ of $G$, the reciprocal of the number of color classes,  $\frac{1}{s}$ or even the reciprocal of the average degree $\overline{d(G)}=\frac{2m}{n}$;   
	\item[--] the \emph{internal density or dyadicity} defined as $2m_i/\kn{c_i}{2}$;  this amounts to choose $w$ proportional to $2(\frac{1}{\kn{c_1}{2}},\ldots,\frac{1}{\kn{c_s}{2}})'$ by the coefficient $\frac{1}{s}$ so to have a score in the range $[0,1]$.
	\end{itemize}
We can also consider weighted version of the score $A(z)$ (which leads to index $a(G,f)$) simply by defining $\Theta(z)=\sum_{i=1}^sw_iz_i$ with $\sum_{i=1}^sw_i=1$. As an illustration consider the homophily-ratio based score which we denote by $R(G,f)$. Hence $R(G,f)=\frac{1}{m}\sum_{i=1}(m_i-\am{M_i})$ and $\mathfrak{r}(G,f)=\frac{1}{m}\sum_{i=1}m_i$ is the descriptive homophily ratio \cite{Zhu_et_al} which, in our model, is viewed as the observed sample of the random variable $\mathfrak{R}=\frac{1}{m}\sum_{i=1}M_i$. In view of \eqref{eq:general_score}, the corresponding homphily index $r(G,f)$ is defined by $$r(G,f)=\text{sgn}(R)\frac{\left(\sum_i(m_i-\am{M_i})\right)^2}{\left(\sum_i(m_i-\am{M_i})\right)^2+\mathbf{1}'\Sigma\mathbf{1}}$$
and, gives an upper on the statistical significance of the tails $\left(\mathfrak{R}\geq \mathfrak{r}(G,f)\right)$ and $\left(\mathfrak{R}<\mathfrak{r}(G,f)\right)$ according to, respectively, if $R(G,f)\geq 0$ or $R(G,f)<0$. 
\mybreak
Besides the quantifiers defined above, there are other possibilities for a global assessing of homophily in the form of estimated significance of an appropriate distance from the expected outcome $\am{M}=(\am{M_1},\ldots,\am{M_s})'$ of the observed outcome $(m_1,\ldots,m_s')$ sampled from the random outcome $M=(M_1,\ldots,M_s)'$. The most celebrated and widely used such distance is the \emph{Mahalanobis distance}. For a random vector $X=(X_1,\ldots X_s)'$ with expectation $\am{X}=(\am{X_1},\ldots,\am{X_s})'$ and definite positive covariance matrix $\Lambda$,  the Mahalanobis distance of $X$ from a particular point $x$ in the support of $X$ is the distance $d_M(X,x)=\sqrt{(x-\am{X})'\Lambda^{-1}(x-\am{X})}$. This distance actually measures how far the point $x$ is from being typical where typical means lying in ellipsoid centered at $\am{X}$ whose axes are the principal components of $X$. Notice that if $X$ is a centered random vector whose marginals have unit variance, then $d_M(X,x)=\|x\|_{\Lambda^{-1}}$, i.e, the Mahalanobis distance of a standardized random vector from a point $x$ in the support of its distribution, coincides with the vector norm induced by the correlation matrix $\Lambda^{-1}$. In particular, if, as usual, $Z$ and $z$ are the $z$ score and the observed $z$ score, respectively, corresponding to the input colored network $(G,f)$, then $d_M(Z,z)=\|z\|_{\Gamma^{-1}}$, where $\Gamma$ is the correlation matrix of $M$. Furthermore, $d_M(Z,z)$ is the Euclidean norm of $z$ whenever $M$ has uncorrelated marginals. For simplicity, we denote $d_M(Z,z)$ simply by $\|z\|$ and refer to it as the Mahalanobis norm of $z$. While it is clear that the higher $\|z\|$ is, the less typical is the observed result $(m_1,\ldots,m_s)$, $\|z\|$ is not a score in the previously defined sense, since $\|z\|$ is not generally monotone even in the absolute values of the coordinates, i.e., it is not true that
$$|z'_i|\leq |z_i|, \,\forall i\in [s]\Rightarrow \|z'\|\leq \|z\|.$$        
However, if $G$ is sparse enough so that $M$ has nearly uncorrelated marginals and $\|z\|$ is close to the Euclidean norm of $z$ (recall Remark \ref{rem:uncorr}), or $G$ is over-dispersed so that $\Gamma$ has non-negative entries, then the above implication holds (at least approximately in the first case). In these cases, $\|z\|$ can be taken as a homophily score, at least in a weaker sense. Similar to the homophily indices defined above, a quantifier of the homophily of $G$ with respect to $f$ based on the weak score $\|z\|$ is obtained by estimating from above the significance of $\|z\|$ under the random coloring model. By a multidimensional version of Chebyshev's inequality (\cite{Chen}) one has 
$$\pra{\|Z\|\geq \|z\|}\leq \frac{s}{\|z\|^2}=\frac{s}{z'\Gamma z},$$   
therefore 
 $$h(G,f)=\max\left\{0,\frac{\|z\|^2-s}{\|z\|^2}\right\}$$
can be taken as a quantifier of homophily in the interval $[0,1]$: values of $h$ close to 1, indicate that the observed outcome is very far from its expected value along the directions most responsible for the variance structure of $M$. However this index alone does not distinguish  between anti-homophilic and homophilic behavior and the analysis must be refined by inspecting the sign pattern of $z$. What is true however, is that if the value of $h$ is positive and small, then the observed outcome is typical.   
\mybreak
Given the results of the previous section, it is clear that both the indices $j_\Theta(G,f)$ (for any choice of weighting vector $w$) and $h(G,f)$ can be computed in optimal $O(s^2+m)$ worst-case time.

\section{Discussion, examples, and applications}\label{sec:applications}
The novelty of our approach is to consider the input colored (partitioned) network as a random object, where the randomness is inherited from random colorings rather than random graphs, and the way we exploit this modeling choice via tail inequalities. While this is nothing more than a reasonable modeling choice, it has some profound implications:
\begin{enumerate}[label=(\Alph*)]
	\item\label{com:A} we gain more information about the significance of the observed outcome when comparing our indices with other consolidated homophily scores. In fact, the significance of the latter cannot generally be assessed per se but rather by the comparisons with other benchmark (possibly randomly generated) instances. Moreover, even locating the values of these scores on a universal scale by comparing them to their respective maxima is a far from trivial problem, the corresponding optimization problems being typically NP-complete (see Theorem \ref{thm:nphardmod} below); rather, our indices designed to provide a measure of their significance by default, always (and tightly) sit in the interval $[-1,1]$; in this sense their scale is universal; 
	    
	\item\label{com:B} due to the simplicity of the random coloring model, we gain the possibility of sharper analysis with respect to random graph models. Recall, that even estimating the edge probabilities in random graph models, such as $\mathcal{G}(d)$, can be challenging \cite{GaoO}. In contrast, the representation of $M$ as a sum of identically distributed Bernoulli random variables allows even the exact computation of the second-order mixed moments of $M$, which allows us to define a whole class of homophily quantifiers, each based on a particular homophily score defined as an affine function of the data or even quantifiers based on the distance from the most typical outcome of the observed one.     
\end{enumerate}
Let us elaborate on these points with an example and real-world applications. Before doing so, we give a formal proof that maximizing the modularity and homophily ratio index over the labeled partition of a given profile are difficult problems. Let $\col$-MODULARITY and $\col$-HOMOPHILY RATIO be the problems of maximizing the modularity and the homophily ratio index, respectively, over the labeled partition of profile $\col$.  
\begin{theorem}\label{thm:nphardmod}
	The decision versions of $\col$-MODULARITY and $\col$-HOMOPHILY RATIO, are NP-Complete problems.
\end{theorem}
\begin{proof}
	Instances of both problems are of the form $(G,\col)$ together with the value $p$ of modularity (respect., homophily ratio) where $G$ is an undirected graph of order $n$ and $\col$ is an $s$-tuple of positive integers that add up to $n$, and we are asked to decide whether a partition of vertices with profile $\col$ exists giving value at least $p$ for the index. Consider the subset of instances where $G$ is cubic, i.e., each vertex has degree 3, and $\col=(\frac{n}{2}, \frac{n}{2})$. Note that $n$ is even because $G$ is a regular graph of odd degree. The set of feasible partitions is thus restricted to the set of partitions into two classes with the same number of vertices. In this case, given \eqref{eq:utile}, $\mathfrak{q}$ and $\mathfrak{r}$ differ by an additive constant. Therefore both indices are maximized when the total number of edges induced by the two classes is maximized and, equivalently, when the number of edges having exactly one end in each of the two classes is minimized. The latter problem is MIN BISECTION ON CUBIC GRAPHS, whose decision version is known to be NP-Complete~\cite{Buitetal}. This proves that $\col$-MODULARITY and $\col$-HOMOPHILY RATIO are hard by restriction, so they are hard in general.
\end{proof}
\subsection{A toy example}\label{subs:toy}
This nice and meaningful example was suggested by an anonymous referee to highlight the difference between modularity and our indices based on tail inequalities. Consider a graph $G$ of order $n$ consisting of $m$ disjoint edges where $m$ is an even number. Hence $n=2m$ and the degree sequence of $G$ is the $2m$-dimensional all ones vector $\mathbf{1}$. Let us analyze the indices over all colorings $f$ of profile $(m,m)$ or, equivalently, over all partitions $\pt$ of $[n]$ into $m$ red vertices and $m$ blue vertices. In this case there is no ambiguity in establishing whether an instance $(G,f)$ is more homophilic than another instance $(G',f')$, where $G=G'$ in the random coloring model, and $f=f'$ in the configuration model. The only problem here, is how to assess whether or not the given instance is homophilic per se. Let $M_r$ and $M_b$ be the random number of homophillic red and blue edges, respectively under the configuration model and the random coloring model. We use the same symbols for random variables on different probability spaces to facilitate comparisons. A random pairing of the vertices of $G$ is called a (simple) configuration. Let $\mathcal{G}(\mathbf{1})$ be the set of graphs on $[n]$ with degree sequence $\mathbf{1}$ and let us equip this set with the uniform distribution (this yields the configuration model in this case). All graphs in $\mathcal{G}(\mathbf{1})$ are simple and the probability of seeing a particular configuration by uniform sampling is $\frac{m!2^m}{2m!}$. For integers $h$ and $k$ such that $0\leq h\leq \frac{m}{2}$ and $0\leq k\leq \frac{m}{2}$, the event $(M_r=h,M_b=k)$ has probability 0 in both probability spaces whenever $h\not=k$. This because edges whose end-vertices have different colors match the same number of vertices in both color classes. Therefore the joint distribution of $(M_r,M_b)$ is supported by the diagonal $\{(k,k), 0\leq k\leq m/2\}$ in both models. Hence
$$\pra{M_1\geq k|\text{conf}}=\pra{M_2\geq k |\text{conf}}=\pra{M_1\geq k, M_2\geq k |\text{conf}}$$
and 
$$\pra{M_1\geq k|\text{rand-col}}=\pra{M_2\geq k |\text{rand-col}}=\pra{M_1\geq k, M_2\geq k |\text{rand-col}}.$$
Actually, it can be shown more, namely,
\begin{equation}\label{eq:toy}
\pra{M_i\geq k|\text{conf}}=\pra{M_i\geq k|\text{rand-col}}=\frac{2^m(m!)^2}{(2m)!}\sum_{t=k}^{m/2}\left(m \atop t,t\right)2^{-2t},\quad i\in \{b,r\}
\end{equation}
where $\left(m \atop t,t\right)$ is the multinomial coefficient of $m$ with parts $t, t, m-2t$.
That is, even though the random variables are defined on different probability spaces, their tails have the same measure so that we can omit the reference to the probability space and denote by $F(k)$ the probability $\pra{M_i<k}=1-\pra{M_i\geq k}$ where $\pra{M_i\geq k}$ is given by \eqref{eq:toy}. Either by the relation above or by direct computation, it follows that 
$$\E_{\text{conf}}(M_r/m)=\E_{\text{conf}}(M_b/m)=\E_{\text{rand-col}}(M_r/m)=\E_{\text{rand-col}}(M_b/m)=\frac{m(m-1)}{2m(2m-1)}=\frac{m-1}{2(2m-1)}.$$   
In fact, since each configuration has probability $\frac{m!2^m}{2m!}$, the probability of seeing an edge connecting two particular vertices is $\frac{1}{2m-1}$ rather than $\frac{1}{2m}$. However, the difference between $\frac{1}{2m-1}$ and $\frac{1}{2m}$ is negligible. In fact $\frac{m-1}{2(2m-1)}$ is close to $\frac{1}{4}$ when $m$ is large. We can thus stick to the definition of $\mathfrak{q}$ given in Section \ref{sec:homophily_in_random_models}. Therefore, by \eqref{eq:modularity}, the modularity of any observed partition $\pt(k)$ with $k$ red edges and $k$ blue edges---recall that no partition with a different number of red and blue homophilic edges can be observed---is
$$\mathfrak{q}(\pt)=2\left(\frac{k}{m}-\frac{1}{4}\right).$$
In conclusion, Newman's modularity interpolates linearly between $-1/2$, corresponding to the perfectly anti-homophilic case, and $1/2$, corresponding the perfectly homophilic case. Clearly the same result would have been obtained by taking expectations with respect to the random coloring model. In this particular case, however, the analysis provided by these indices is not sharper than what the simple homophily ratio $\mathfrak{r}$ provides: $\mathfrak{r}(\pt(k))=\mathfrak{q}(\pt(k))+1/2$. On the other hand, by looking at Figure \ref{Fig:toy}, one can appreciate the information gain provided by the indices based on concentration inequalities. The figure refers to a matching with $1000$ vertices, $500$ edges and coloring profile $(500,500)$. Since in every coloring of profile $(500,500)$ the number of  red edges is the same as  the number of blue edges, all the indices $\mathfrak{q}$, $\mathfrak{r}$, and $a$, are actually real valued functions defined on the integers between $0$ and $250$. For $\mathfrak{q}$ and $\mathfrak{r}$ this is clear by their definition, while for $a$ it suffices to observe that the arithmetic mean of the observed $z$-scores corresponding to the observed outcome $(k,k)$, coincides with the marginal $z$-score of the red edges (or, equivalently, of the blue edges). In particular, the plot of 
$\mathfrak{q}$ and $\mathfrak{r}$ are (parallel) straight lines, while $a$ is of the form $\frac{z^2(k)}{1+z^2(k)}$ with $z(k)$ being the marginal $z$-score corresponding to $k$ red edges. This $z$-score is itself an affine function of $k$. In the figure we plot the corresponding values, as $k$ runs between $0$ and $250$, against the plot of the (exact) probability $F(k)$. 
\begin{figure}[H]
\begin{center}
\includegraphics[width=15cm]{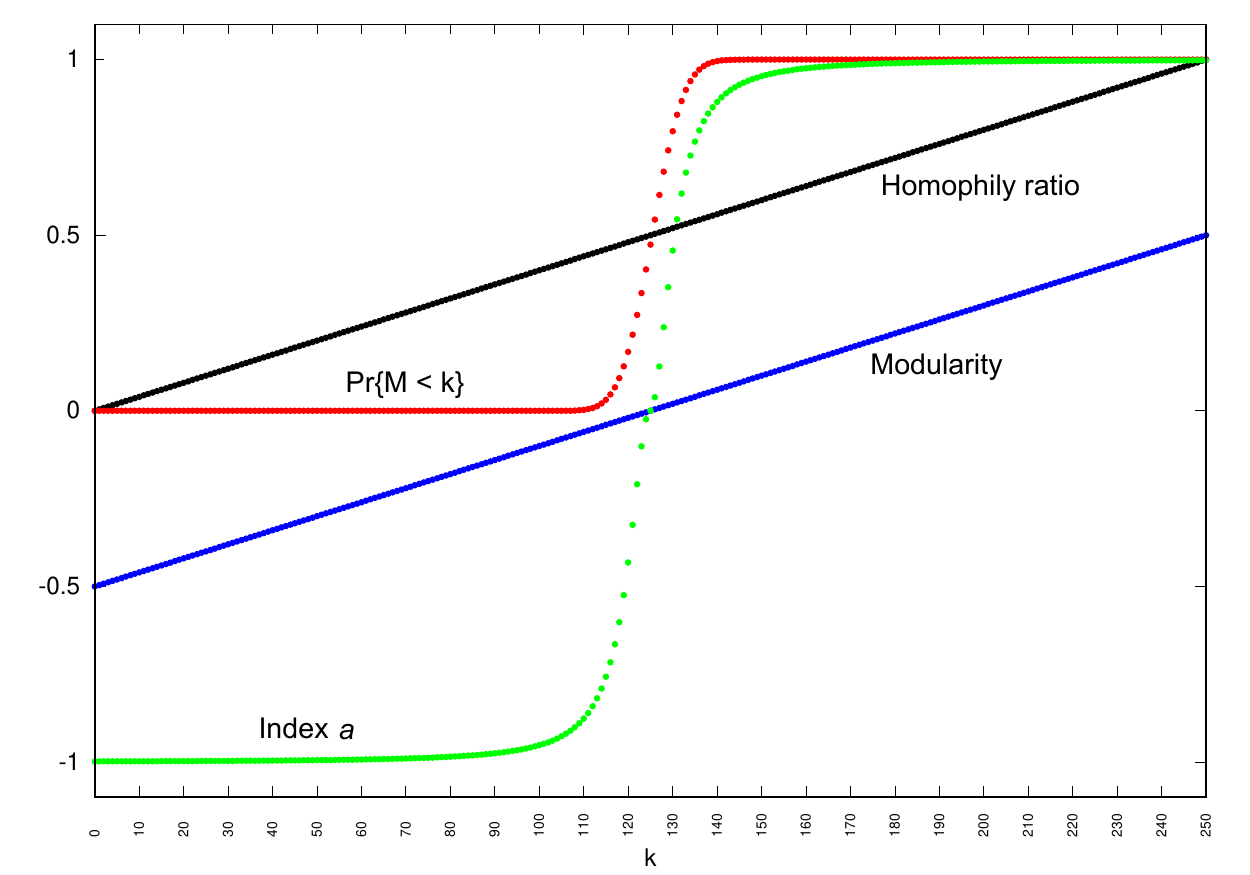}
\caption{Theoretical probability $\pra{M < k}$ (red) for the matching example in Section~\ref{subs:toy} compared against Homophily Ratio (black),  Modularity (blue) and Index $a$ (green). Values are reported as a function of $k$. Experiments refer to an instance with 500 edges and coloring profile (500, 500).}
\label{Fig:toy}
\end{center}
\end{figure}

By inspecting the plots of $F(k)$ it is clear that values of $k$ (and hence the observed outcomes of the form $(k,k)$) outside the interval $[110,140]$ are extremely unlikely, meaning anti-homophily (left tail) or homophily (right tail). Neither $\mathfrak{q}$ nor $\mathfrak{r}$ can detect this phase transition, while index $a$ clearly does. In conclusion, the situation is much the same as deciding whether the number of heads in a long series of coin tosses is consistent with the assumption that the coin is fair: $\mathfrak{q}$ and $\mathfrak{r}$ would only count the fraction of heads, while $a$ would estimate the likelihood of the sequence.

\subsection{Measuring homophily on real networks.}\label{subsec:real_net} 
If we want to decide whether a particular categorical property correlates with the structure of the network, we simply encode the property as a coloring and apply one of the proposed indices. Here we concentrate on $a(G,f)$ and $r(G,f)$. To illustrate this application we implemented a prototype function in Python 3 for computing these indices and run it on five instances of the form $(G,f)$. In the first four instances $G$ is a Protein Protein Interaction (PPI) network: each vertex of the network represents a protein while edges represent interactions between proteins in different organisms. More specifically, for the first three instances, $G$ is one of \emph{Treponema pallidum}, \emph{Escherichia coli} and \emph{Saccharomyces cerevisiae}---retrieved from STRING data-base (\texttt{https://string-db.org/})---while in the fourth instance, $G$ is the \emph{Homo sapiens} PPI network. In the last instance, $G$ is the network \emph{Pokec} which describes symmetric interactions among users of a popular Slovak social network. 
The coloring $f$ in the first three instances associates each vertex with its functional class according to the standard taxonomy of the NCBI database (ftp://ftp.ncbi.nih.gov/pub/COG/COG/). As explained in detail in \cite{apollonio2022}, the range of $f$ consists of 19 colors rather than 24. In the fourth instance, $f$ associates each protein with the chromosome the related gene belongs to. Finally, in the last instance, $f$ associates each user with one out of the five classes of (declared) age. For biological reasons, the first four networks are inherently homophilic with respect to the chosen colorings and there are reasons to expect that the same is true for the social network. The results of our experiments, summarized in Tab.~\ref{tabellareti}, confirm these expectations. The second and the third columns of Tab.~\ref{tabellareti} correspond to the number of vertices $n$ and the number of edges $m$ of $G$, respectively; the other columns of the table report the values of the indices used in the comparison: the random homophily ratio based index $r(G,f)$, the arithmetic mean of the $z$-scores based index $a(G,f)$, the descriptive homophily ratio $\mathfrak{r}(G,f)$, and Newman's modularity $\mathfrak{q}(G,f)$. Actually, the values of $r(G,f)$ and $a(G,f)$ are complemented to 1 to give an immediate perception of the significance of the value of the corresponding statistic. For example, the last row of the fourth column of the table reads: the probability of observing, by chance, a coloring with a higher homophily ratio than the observed one is less than 1 in half a million---that is, much less than the significance level typically set at 5\%---. Since locating the maximum of $\mathfrak{r}$ and $\mathfrak{q}$ is a difficult optimization problem (see Theorem~\ref{thm:nphardmod} and recall \ref{com:A}), the significance of  these indices cannot be  assessed by comparisons to their respective maxima but rather by a comparison with the results in Tab.~\ref{tabellaretirandom} where we reported the values of all of the indices averaged over five randomly generated instances $(G,f_1),\ldots (G,f_5)$ where $G$ is kept fixed and the $f_i$'s are picked independently at random among the colorings of $G$ with the same profile as the input coloring $f$. Even without knowing the actual range of $\mathfrak{r}$ and $\mathfrak{q}$, it can be seen from Tab.~\ref{tabellaretirandom} that all four considered indices behave as expected, i.e., none of them deviates significantly from its expected reference value. Also the comparison between the last column of the two tables shows that the index $\mathfrak{q}$ (as well as $\mathfrak{r}$) suggests homophily with respect to the corresponding coloring of all five networks except Homo sapiens. This network is not considered to be significantly homophilic w.r.t.\ the given colouring by both the modularity index and the descriptive homophily ratio, while it is considered homophobic by both of our indices, consistent with the underlying biological reasons. Also note that although $n$, $m$ span a wide range and the edge density differs by up to three orders of magnitude, our indices assume fairly comparable values, confirming the consistency of our method. The time needed to compute our indices is negligible with respect to the I/O time for reading the graphs, confirming that both the $a(G,f)$ and $r(G,f)$ indices can be computed in linear time without using any special data structure, provided that the number of colors is $O(\sqrt{m})$ (which is not restrictive in concrete applications). 
\begin{table}[H]\label{tabellareti}
	\begin{center}
		\begin{tabular}{|c|r|r|r|r|r|r|r|r|r|r|}
			\hline
			\textbf{network}&\multicolumn{1}{|c|}{\textbf{vertices}}&\multicolumn{1}{|c|}{\textbf{edges}}&$1-r(G,f)$&$(1-a(G,f))10^6$&$\mathfrak{r}(G,f)$&$\mathfrak{q}(G,f)$\\
			\hline
			\emph{Treponema pallidum}
			&894
			&8,157
			&0.0014
			&160.8
			&0.545
			&0.3538\\
			
			\emph{Escherichia coli}
			&4,020
			&29,748
			&0.0015
			&11.1
			&0.507
			&0.3332\\
			
			\emph{Saccharomyces cerevisiae}
			&6,157
			&119,051
			&0.1233
			&18.0
			&0.498
			&0.2138\\
			
			\emph{Homo sapiens}
			&27,342
			&777,098
			&0.0025
			&1007.3
			&0.049
			&-0.0011\\
			
			\emph{Pokec}
			&1,212,349
			&8,320,600
			&1.6582e-06
			&3.3
			&0.466
			&0.1629\\
			\hline
			
			\hline
			
			\hline
		\end{tabular}
	\end{center}
	\caption{Indicators for  the five considered instances}
\end{table}

\begin{table}[H]\label{tabellaretirandom}
	\begin{center}
		\begin{tabular}{|c|r|r|r|r|r|r|r|r|}
			\hline
			\textbf{network}&$1-r(G,f)$&$(1-a(G,f))10^6$&$\mathfrak{r}(G,f)$&$\mathfrak{q}(G,f)$\\
			\hline
			\emph{Treponema pallidum}
&0.6715
&457,697
&0.161
&-0.00295\\
\emph{Escherichia coli}
&0.6487
&499,182
&0.467
&0.00028\\
\emph{Saccharomyces cerevisiae}
&0.7734
&539,713
&0.282
&-0.00013\\
\emph{Homo sapiens}
&0.5784
&709,707
&0.107
&-0.00024\\
\emph{Pokec}
&0.5720
&770,348
&0.257
&0.00003\\
			\hline
			
			\hline
			
			\hline
		\end{tabular}
	\end{center}
	\caption{Indicators for the randomised colouring of the  considered instances, averaged each on 5 random colorings}
\end{table}
\section{Conclusion}\label{sec:concl}
We proposed a novel inferential statistical approach, based on second-order tail inequalities, to quantify network homophily for an input $(G,\pt)$ consisting of a undirected graph $G$ and a given partition of its vertex set. After adhering to the accepted paradigm that whatever network homophily means, it is measured by a non-decreasing function of the edge densities of the subgraphs induced by the classes of $\pt$, we proposed to look at any such function, referred to as \emph{descriptive homophily index}, as a random variable defined on the probability space of the labeled partitions of the same profile as $\pt$ endowed with the uniform measure. This is the random coloring model introduced in \cite{apollonio2022} and formally refined in the present paper. Within the random coloring model, values of descriptive indices are just observed samples from the corresponding random variables described by picking labeled partitions at random. In light of this view, we proposed to quantify network homophily as the (signed) statistical significance ($p$-value) of the observed value under the null hypothesis of absence of homophily. The $p$-value of the observed value is the probability of the appropriate (right or left) tail of the corresponding random variable, namely, the probability of observing a more extreme value than the observed one. We showed that the problem of computing the $p$-value is hard to approximate, since the corresponding optimization problem is not even in the APX class. However we gave upper bounds on this $p$-value by mimicking the three sigma-rule of thumb, namely by resorting to tail inequalities such as Chebyshev's and Cantelli's inequalities. While from the one hand we have no control on the quality of the approximation of the $p$-value via these inequalities, on the other hand, in view of the inapproximability result, this is the best we can hope for unless $P=NP$. Estimating the $p$-value via tail inequalities requires the knowledge of covariance matrix of the random vector $M$, which was not previously known within the random coloring model. Due the simplicity of the model, in this paper we closed this gap by giving a closed form expression for the covariance matrix of $M$. In particular, we discovered that the correlation among the marginals of the random vector $M$ is ruled by a parameter $\gamma(G)$ whose sign is a graph property (hence $\gamma(G)$  is a graph invariant) and whose absolute value dictates the intensity of the correlations---for sparse enough graphs, the marginals are nearly uncorrelated---. Interestingly, this parameter depends on the input partition only through its profile and depends on the network only through its degree sequence,
 implying that both the covariance matrix and its inverse can be computed very efficiently leading to an overall complexity of $O(s^2+m)$ worst-case time for computing all the proposed indices. Our indices, that range in interval $[-1,1]$ for any graph, supply some gain of information when compared to other well established indices. The situation is much the same as deciding whether the number of heads in a long series of coin tosses is consistent with the assumption that the coin is fair. For this purpose, while classical indices only count  the fraction of heads, our indices estimate the likelihood of the sequence. This is confirmed by the case of the Homo sapiens PPI network with respect to the location of corresponding genes on chromosomes. For biological reasons, the Homo sapiens PPI network is inherently homophilic with respect to the given partition. However, neither Newman's modularity nor the homophily ratio (nor other descriptive indices) are able to detect this property, whereas ours are.   

\paragraph{Ackowledgments} We can say that we were very lucky with both referees. In dealing with their precise, detailed and decisive criticisms, we had to go deeper to understand the problem better. As a result, the presentation has definitely improved in both form and substance. We owe to one of the two anonymous referees the artificial example of Section \ref{subs:toy}, while we owe to both of them the improved computational results of Section \ref{subsec:real_net}.

\end{document}